\documentclass{amsart}[14pt]
\usepackage{amssymb, amsfonts, amsmath
}
\usepackage{hyperref}
\usepackage{cite}

\newtheorem{theorem}{Theorem}[section]
\newtheorem{lemma}[theorem]{Lemma}
\newtheorem{proposition}[theorem]{Proposition}
\newtheorem{corollary}[theorem]{Corollary}

\theoremstyle{definition}
\newtheorem{definition}[theorem]{Definition}
\newtheorem{example}[theorem]{Example}
\newtheorem{notation}[theorem]{Notation}

\theoremstyle{remark}
\newtheorem{remark}[theorem]{Remark}

\newcommand{\cA}{{\mathcal A}}
\newcommand{\cB}{{\mathcal B}}
\newcommand{\cC}{{\mathcal C}}
\newcommand{\cD}{{\mathcal D}}

\newcommand{\cF}{{\mathcal F}}
\newcommand{\cH}{{\mathcal H}}

\newcommand{\cK}{{\mathcal K}}
\newcommand{\cL}{{\mathcal L}}

\newcommand{\cO}{{\mathcal O}}
\newcommand{\cP}{{\mathcal P}}

\newcommand{\bC}{{\mathbb C}}

\newcommand{\bR}{{\mathbb R}}

\renewcommand{\d}{\delta}
\newcommand{\D}{\Delta}

\renewcommand{\O}{\Omega}

\newcommand{\s}{\sigma}

\newcommand{\<}{\langle}
\renewcommand{\>}{\rangle}

\numberwithin{equation}{section}

\begin{document}

\title[Quantum Hamming metrics ]
{Quantum Hamming metrics}

\author{Marc A. Rieffel}
\address{Department of Mathematics\\
University of California\\
Berkeley, CA\ \ 94720-3840}
\email{rieffel@math.berkeley.edu}
   


\large{

\begin{abstract}
Given the set of words of a given length for a given alphabet, the Hamming metric between two such words is the number of positions where the two words differ. A quantum version of the corresponding Kantorovich-Wasserstein metric on states was introduced in 2021 by De Palma, Marvian, Trevisan and Lloyd. For the quantum version the alphabet is replaced by a full matrix algebra, and the set of words is replaced by the tensor product of a corresponding number of copies of that full matrix algebra. While De Palma et al. work primarily at the level of states, they do obtain the corresponding seminorm (the quantum Hamming metric) on the algebra of observables that plays the role of assigning Lipschitz constants to functions. A suitable such seminorm on a unital C*-algebra is the current common method for defining a quantum metric on a C*-algebra. 

In this paper we will reverse the process, by first expressing the Hamming metric in terms of the C*-algebra of functions on the set of words, and then dropping the requirement that the algebra be commutative so as to obtain the quantum Hamming metric. From that we obtain the corresponding Kantorovich-Wasserstein metric on states. Along the way we show that many of the steps can be put in more general forms of some interest, notably for infinite-dimensional C*-algebras.
\end{abstract}

\maketitle


\section{Introduction}

In  2021 De Palma et al. published a paper \cite{DMTL} in which they construct  quantum versions of the Kantorovich-Wasserstein metrics on probability measures corresponding to Hamming metrics. They present many basic properties of these quantum versions, and indicate strong connections with a variety of important topics in quantum information theory. Their paper has spawned a large number of papers that apply the ideas in their paper in many directions. 
See \cite{Btt, dPKP, 
dPR, dPT, HlR, KdML, MoP, QCZ, RzF, TrvL, TrvI, ZDK, HRF} 
and the references they contain.
In the course of their paper the authors obtain a quantum version of the Hamming metrics themselves.

In the present paper we reverse the process, and start with 
Hamming metrics themselves, and provide a natural path 
that leads to the definition of quantum Hamming metrics. 
With this definition in hand, one is led quite naturally to the
corresponding quantum Kantorovich-Wasserstein metrics on 
the state spaces of the underlying quantum systems, 
which in \cite{DMTL} are taken as the \emph{definition} of the quantum 
Kantorovich-Wasserstein metrics (and which they call
``The quantum Wasserstein distance of order 1''). 
In the present paper we put 
various of the steps into suitable more general contexts, 
which provide more perspective for the steps. In particular,
we show that many (but not all) of the steps work well
for arbitrary, possibly infinite-dimensional, unital C*-algebras. 
We also show how the 
quantum Hamming metrics can be expressed in terms of
Dirac operator type structures.

Here is a quick outline of our approach. We note first that the Hamming 
metric comes from first putting the complete-graph metric on each 
alphabet in use, and then using the corresponding  sum-metric on 
the cartesian product of the alphabets.
We then show how to express these two steps in terms of 
the corresponding algebras
of functions (the algebras of ``observables''), so that the Hamming metric
is expressed entirely in terms of these algebras. Since this reformulation 
does not mention points of spaces, nor use the commutativity of the algebras, 
we can replace the algebras with ones that are not commutative, such as full 
matrix algebras, to obtain our quantum version of the Hamming metrics. The 
metric data is then given by a seminorm that plays the role of attaching the 
Lipschitz constants to functions on a metric space. We also show
how to express the metric data in terms of ``Dirac operators'', from which
a few additional properties become evident. From the
formulas for the quantum Hamming metrics we then obtain the formulas
for the corresponding Kantorovich-Wasserstein metrics on states, 
that in \cite{DMTL} are taken as the \emph{definition} of the quantum
Wasserstein-1 metrics on states.

Necessarily many of the results in the present paper are 
counter-parts, perhaps with some
generalization, of results in \cite{DMTL}.

Kuperberg and Weaver \cite{KpW} have introduced a different definition
of quantum metric spaces, motivated by quantum error correction, and
formulated in the setting of von Neumann algebras.
They call their metric structure a ``W*-metric''.
 We sketch here what
it looks like just in the finite-dimensional case, which is sufficient for their
quantum Hamming metric, for which see \cite{KpW} and 
example 11.30 in \cite{Wvr4}. 
For a finite-dimensional Hilbert space $\cH$ and  a C*-algebra $\cA$ of 
operators on it, they define a W*-metric on $\cA$
to be a filtration $V_t$ of operator systems (i.e. self-adjoint subspaces
containing the identity operator) in $\cB(\cH)$ for
$t$ ranging over a finite set of non-negative numbers.
The filtration is required to satisfy the properties that 1) $V_0$ is
the commutant of $\cA$ in $\cB(\cH)$, 2) $V_s V_t \subseteq V_{s+t}$.
The intuition is that $V_t$ consists of operators that displace
mass by at most distance $t$ (or introduce at most $t$ errors). 
The term ``finite-propagation 
operators'' is sometimes used. 
The authors show that their W*-metrics can sometimes lead to 
C*-metrics as used in our present paper, and that Dirac-type
constructions lead to W*-metrics. But very little is said about
states, so the W*-metric approach appears to go in a 
direction quite different than that of the 
paper \cite{DMTL} of De Palma et al.
that we study here.

In \cite{DMTL} the authors present many important properties
and applications of their
Wasserstein-1 distance on states, which, as indicated
above, have been widely explored and extended in other papers both
by the authors and by many other researchers. We will
not discuss any of this in the present paper, but there
is much that well warrants
exploration in the style of the present paper.

\tableofcontents


\section{Hamming metrics}
\label{secprim}

In this section we begin the discussion of ordinary 
Hamming metrics, aiming to put them in
a form suitable for suggesting the quantum versions 
that we will discuss later.
The setting is as follows.
We are given a finite set with $n$ elements, $I$, 
whose elements we use as indices.
For each $i \in I$ we are given a finite set, $X_i$,
called the alphabet at position $i$. 
(For the quantum Hamming metric, for each $i$ 
our alphabet can be a quantum
finite set, that is, a full matrix algebra 
$\cA_i = M_{m_i}(\bC)$ with its structure as a C*-algebra.) 
Often it is assumed that all
these alphabets are the same set, but we do not need
to make this assumption. We then consider ``words'' 
of length $n$ with the $i$th entry taken from $X_i$
for each $i$. The set of these words can be
identified with the product set $X = \Pi_i X_i$.
 (In the quantum setting there are no words, but
 since $C(X) = \otimes C(X_i)$,
in the quantum setting we will use
the C*-algebra $\otimes \cA_i$.) The 
Hamming metric on $X$ is the metric that sets
the distance between two words to be the
number of indices $i$ for which the $i$-th entry of
the two words is different.
This suggests that on each $X_i$ we let $d_i$
be the metric that gives distance 1 to all pairs
of distinct elements of $X_i$. 

Given metrics on several sets, there is a large
number of ways of using them to define a metric 
on the product of those sets. The definition of
the Hamming metric suggests that we use
the ``sum-metric'', that is
the sum of the metrics. Thus we define
the metric $d_H$ on $X$ by
\[
d_H(a, b) = \sum_i d_i(a_i,b_i)
\]
for all $a, b  \in X$. It is easily checked that
this is indeed the Hamming metric. For the more general
case in which the $X_i$'s are general metric
spaces (usually compact in later sections), with arbitrary metrics
$d_i$, we will denote the sum-metric on their
product by $d_S$.

To try to see what the quantum version of the
Hamming metrics might be, we follow the often-used
path of expressing the situation in
terms of the of the algebra of functions on
$X$, and then dropping the requirement that
the algebra be commutative.

For any compact metric space $(X,d)$ we denote
the algebra of continuous $\bC$-valued functions
on $X$ by $C(X)$. There are quantum physics experiments
which are considered to indicate that nature knows
about $i = \sqrt {-1}$ and that $i$ is necessary for
modeling quantum phenomena. See \cite{RAN}
and the references in \cite{LnP}.) For this reason
we will usually use $\bC$-valued functions, and
matrix algebras over $\bC$.
For any $f \in C(X)$ we let
$L^d(f)$ be its Lipschitz constant, defined by
\begin{equation}
\label{lipsch}
L^d(f) = \sup \{|f(x) - f(y)|/d(x,y): x, y\in X, \quad x \neq y\}
\end{equation} 
(which may be $+\infty$).
We let $Lip_d(X)$ be the dense sub-algebra of $C(X)$
consisting of functions $f$ such that $L^d(f) < \infty$, that is, are
Lipschitz for $d $. Thus $L^d$, as defined in equation
\ref{lipsch}, restricts to a finite-valued seminorm on $Lip_d(X)$. 
But aside from the fact that $L^d$ can take value $+\infty$
on $C(X)$, 
it has in general quite nice properties
\cite{R4, R5, R21}. These properties
motivate the following definition (see definition 4.1 of \cite{R21}), 
which will be the context for
our treatment of the quantum situation. A proof of property
4 below for compact metric spaces will follow 
from our later Theorem \ref{predual}.

\begin{definition} 
\label{deflip}
Let $\cA $ be a unital C*-algebra, 
and let $ L $ be a seminorm on $\cA $ 
that may take the value $+\infty $. We say that
$ L$ is a \emph{C*-metric} if it satisfies the following properties: 
\begin{enumerate}
\item For any $a\in \cA $ we have 
$ L (a) = 0 $ if and only if $a \in \bC 1_\cA $, and
$L$ is a $*$-seminorm in the sense that
$ L (a^*) = L (a) $.
\item $\cL_L := \{a: L (a) < +\infty\}$ is dense in $\cA $ 
(and so is a dense $*$-subspace of $\cA$).
\item $L$ satisfies the Leibniz inequality
\begin{equation*}
\label{eqleib}
L(ab) \leq L(a)\|b\| + \|a\| L(b)  
\end{equation*}
for all $a,b \in \cA$. (So $\cL_L$ is a $*$-subalgebra of $\cA$.)
\item Define a metric, $d^L $, on the state space
$ S (\cA) $ of $\cA$ by
\begin{equation}
\label{distance}
d^L (\mu, \nu) =\sup \{|\mu (a) -\nu (a)|: a \in \cA, \ \ L (a)\leq 1 \}  .
\end{equation}
 We require that the topology 
on $ S (\cA) $ determined by this metric agrees with the 
weak-$*$ topology (so it is compact).
\item $L$ is lower semi-continuous with respect to the 
operator norm.
\end{enumerate}
We will call a pair $(\cA, L)$ with $\cA$ a unital C*-algebra
and $L$ a C*-metric on $\cA$ a \emph{compact C*-metric space}.
If a Leibniz seminorm also satisfied the property that if 
$a \in \cA$ is invertible in $\cA$
then $L(a^{-1}) \leq \|a\|^{-2}L(a)$, as discussed 
in sections 1 and 2 of \cite{R21}), then we say that $L$ is
\emph{strongly Leibniz}.
\end{definition}

\begin{remark}
\label{remar}
A priori $d^L$ can take the value $+\infty$. 
But that can not happen if property (4) holds.
Also, because $L$ is a $*$-seminorm, 
it suffices to take the supremum in property (4)
only over self-adjoint $a$'s. To see this, note that
if we let $\phi = \mu - \nu$ then $\phi(a^*) = \overline{\phi(a)}$ for
all $a \in \cA$. Then for any given $a$, we can multiply it
by a complex number of modulus 1 in such a way that
$\phi(a) \geq 0$ and $L(a)$ is unchanged. Then
$\phi((a  + a^*)/2) = \phi(a)$.  
\end{remark}

One of the consequences of Theorem \ref{predual} below
is a proof of the well-known fact that
for any compact 
metric space $(X,d)$, the pair $(C(X), L^d)$
is a compact C*-metric space.
 
 Accordingly, for $X = \Pi_i X_i$ with its
 Hamming metric $d_H$ as defined earlier, we will
 provisionally take $L_H$,
as defined by equation \ref{lipsch}, to
 be the corresponding C*-metric on $C(X)$.
 But we will see later that when working with $\bC$-valued
 functions this is not entirely satisfactory.
 
 We need to explore some of the properties of Hamming metrics. 
 We begin with one simple property. For any
 finite set $X$ let $d^1$ be the metric that gives distance 1
 to any pair of distinct points, and let $L^1$ be the
 corresponding C*-metric on $C(X)$.
 
 \begin{proposition}
 \label{bound}
 Let $X = \Pi_{ i= 1}^n X_i$ as before, where the $X_i$'s
 are finite sets.
 Then the diameter of $(X, d_H)$ is no more than $n$, and
$ L^1 \leq nL_H \leq nL^1$ \ . If each $X_i$ contains at
least 2 points, then the diameter of $(X, d_H)$ is $n$.
 \end{proposition}
 
 \begin{proof}
 A moment's thought shows that the metric $d_H$
 takes values between $1$ and $n$. Thus for
 any $f\in C(X)$ and any $x, y \in X$ we have
 \[
  |f(x) - f(y)| \geq |f(x) - f(y)|/d_H(x,y)  \geq (1/n) |f(x) - f(y)|   ,
 \]
 from which the first statement follows immediately. For
 the second statement, if each $X_i$ contains at least 2
 points, then there are at least two elements of $X$ that
 differ from each other in $n$ places, and so have
 Hamming distance $n$.  
 \end{proof} 
 

\section{Sum-metrics}
\label{sump}

We examine first
sum-metrics on products of compact metric spaces.
So for each $i \in I$ let $(X_i, d_i)$ be a compact
metric space, let  $X = \Pi_i X_i$ with the product topology, 
and let $d_S$
be the sum-metric on $X$. For each $i \in I$
define $L_i$ on $C(X)$ by
\begin{equation}
\label{partial}
L_i(f) = \sup \{|f(x) - f(y)|/d_i(x_i,y_i): x \neq y\in X, \ \ x_j 
=  y_j \  \mathrm{for} \ j \neq i\}.
\end{equation}
This can be considered to be a rough analog of the partial
derivative of $f$ in the $i$ direction. As usual, $L_i(f)$
can be $+\infty$. We record here an easily verified fact that
will be quite important later:

\begin{proposition}
\label{nullj}
With notation as above, let $\tilde X_i = \Pi \{X_j: j \neq i\}$, 
so that $X = X_i \times \tilde X_i$. 
Then the null-space of $L_i$ 
consists of the functions that are constant
in the $i^{th}$ variable, and so this null-space is $C(\tilde X_i)$,
where  $C(\tilde X_i)$ is
viewed as a subalgebra of $C(X)$ via the natural projection
from $X$ onto $\tilde X_i$.
\end{proposition}

The following result, and later its quantum version, 
are central to this paper.

\begin{theorem}
\label{summet}
For each $i \in I$ let $(X_i, d_i)$ be a compact
metric space, let  $X = \Pi_i X_i$, and let $d_S$
be the sum-metric on $X$.  Let $L^{d_S}$ be the
Lipschitz seminorm for $d_S$, and for each $i$
let $L_i$ be as defined in equation \eqref{partial}. Then
\[
L^{d_S}(f) = \max\{ L_i(f): i \in I\}
\]
for all $f \in C(X)$. In particular, $f \in Lip_{d_S}(X)$ exactly if
$L_i(f) < \infty$ for each $i \in I$.
\end{theorem}

\begin{proof}
We treat first the case in which $I$ contains only two elements, 1 and 2,
and so $f$ is a function of two variables. 
For any  $s, t, u, v \in \bR^+$ with $u, v$ strictly positive we have
the inequality
\[
\frac{s+t }{u+v} \leq \frac{s}{u} \vee \frac{t}{v}   ,
\]
where ``$\vee$'' means ``maximum''.
The proof of this inequality is straight-forward, especially if one begins by
multiplying the inequality by $v/t$.

Let $f \in C(X_1 \times X_2)$.
With evident notation,
for all $a,b \in X_1$ and $p,q \in X_2$ with $a \neq b$ and $p \neq q$ we have
\[
\frac{|f(a,p) - f(b, q)|}{d_S((a,p), (b,q))} \leq
\frac{|f(a,p) - f(b,p)|+|f(b,p)- f(b,q)|}{d_1(a,b) + d_2(p,q)},
\]
which by the above inequality is
\[
\leq \frac{|f(a,p) - f(b,p)|}{d_1(a,b)} \vee \frac{|f(b,p)- f(b,q)|}{d_2(p,q)} .
\]
If instead $p=q$ then we have an evident simplified version of the inequality,
and similarly if $a = b$. 
Since the term after the last inequality is clearly no bigger than
$L^{d_S}(f)$, the proof for two spaces is complete.

We now deal with the general case by induction on the 
number, $n$, of elements in $I$. 
 We have just given
a proof for the case of $n=2$. We now assume that $n \geq 3$,
and that the conclusion holds for $n-1$. Choose an index $i_0$,
 and let $J=I \setminus \{i_0\}$. 
Let $Y=\Pi\{X_i: i \in J\}$, with corresponding sum-metric $d_S^J$.
Then $X = Y\times X_{i_0}$, and, importantly, $d_S$ is the sum-metric
of $d_S^J$ and $d_{i_0}$. Thus from the case for $n=2$ treated above,
we see that for any $f \in C(X)$ we have
\[
L^{d_S}(f) = \sup\{\frac{|f(a,p) - f(b,p)|}{d_S^J(a,b)} \}
\vee \sup \{ \frac{|f(b,p)- f(b,q)|}{d_{i_0}(p,q)}\} 
\]
where ``sup'' is over all $a,b \in Y$ and all $p, q \in X_{i_0}$.
But from the induction hypothesis that the proposition is true
for $n-1$, it is easily seen that the term
\[
\sup{\{\frac{|f(a,p) - f(b,p)|}{d_S^J(a,b)}: a, b \in Y, p \in X_{i_0}\} }
\]
is equal to $ \max\{ L_j(f): j \in J\}$,
while the term
\[
\sup\{ \frac{|f(b,p)- f(b,q)|}{d_{i_0}(p,q)}\} : b \in Y, \ \ p, q \in X_{i_0}\}
\]
is easily seen to be $L_{i_0}(f)$. Thus the ``max'' of these two terms is
$ \max\{ L_i(f): i \in I\}$ as desired
\end{proof}

It seems to me quite possible that the above theorem already appears
somewhere in the huge literature about metric spaces, but so far
I have not found it there. This comment also applies to some of
our later results about metric spaces.

Notice that when $X$ is finite, 
calculating $L^{d_S}(f)$ by calculating the $L_i(f)$'s and
using the formula of the above theorem requires substantially less
computation than is needed to calculate the Lipschitz constant of $f$
for an arbitrary metric on $X$, and the calculations of the $L_i(f)$'s
can be done in parallel over the various $i$'s.

Sum-metrics have another nice property whose quantum version is of 
some importance, and a special instance of which occurred in the 
above proof.

\begin{proposition}
\label{sumadd}
Let $\{(X_i, d_i)\}$ and 
$\{(X'_j, d'_j)\}$ be two finite collections of 
compact metric spaces. Let $X = \Pi X_i$, equipped with sum-metric
$d_S$, and let $X' = \Pi X'_j$, equipped with sum-metric $d'_S$.
Let $d''_S$ be the sum-metric on $X \times X'$ for $d_S$ and $d'_S$.
Let $Z = \Pi ( \{X_i\}\cup \{X'_j\})$, equipped with the sum metric $d^Z_S$
for $\{d_i\}\cup \{d'_j\}$. Then
\[
(Z, d^Z_S) = (X\times X', d_S'')
\]
for the evident identifications.
\end{proposition}

The proof is straight-forward. (To simplify notation we will
often write, as above, $\{X_i\}$ instead of $\{X_i\}_{i \in I}$,
and similarly for other indexed collections.)


\section{Diameters of compact C*-metric spaces}
\label{diamstar}

Quantum physicists are very interested in states, which are the generalization
of probability measures to the setting of non-commutative C*-algebras. That is, states
are the positive linear functionals of norm 1 on a C*-algebra.
For mathematical models of quantum systems, the mathematical states
usually correspond to the possible physical (mixed) 
states of the physical quantum system
being modeled. Consequently
physicists are very interested in various notions of distance between states.
In the setting of compact metric spaces one has the
well-known definition of the Monge-Kantorovich-Wasserstein metric
on the set of probability measures on a given compact metric space.
It is also variously called the ``earth-mover metric" or the ``transportation-cost
metric'' \cite{Wvr4}. We will be exploring that metric for the case
of sum-metrics, and then in our discussion of quantum Hamming metrics.

 But in several places in the next sections we will be concerned with
 the radius and diameter of state spaces for various metrics defined
 by seminorms. 
As a convenience to the reader, we recall now some of
the discussion in section 2 of \cite{R5} concerning the 
diameter and radius of
a state space, and we relate it to compact metric spaces. 
Here, and throughout
this paper, 
for any C*-algebra $\cA$ we set $\cA^0 = \cA/\bC 1_\cA$, and
we let $L^0$ be the seminorm on 
$\cA$ that is the pull-back to $\cA$ of the quotient norm
on $\cA^0$, that is;
\begin{notation}
\label{Lzero}
$L^0(a) = \inf \{\|a  - z1_\cA\|: z \in \bC\}$
for any $a \in \cA$. 
\end{notation}

We will use the well-known fact (see lemma 2.1
of \cite{R5}, a consequence of the Jordan decomposition
for self-adjoint linear functionals on $\cA$ \cite{Pdr}) 
that the set $\{\mu - \nu: \mu, \nu \in S(\cA)\}$
coincides with the the ball $D_2$ of radius 2 in 
the annihilator of $1_\cA$ in the
dual of the space of self-adjoint elements of $\cA$. We do not
require here that the seminorms $L$ satisfy conditions (3), (4) or (5)
of Definition \ref{deflip}. 

\begin{proposition}
\label{diamet}
Let $\cA$ be a unital C*-algebra, and let $L$ be
a seminorm on $\cA$ that satisfies properties
(1) and (2) of Definition \ref{deflip}. Then it determines a metric $d^L$
on $S(\cA)$ exactly as in the formula \ref{distance}, except 
that the metric may take the value $+\infty$. Then the following
statements are equivalent for any given $ r \in \bR^+$:
\begin{enumerate}
\item  For all $\mu, \nu \in S(\cA)$ we have $d^L(\mu, \nu) \leq 2r$.
\item  For all self-adjoint $a \in \cA$ we have $L^0(a) \leq rL(a)$.
\end{enumerate}
\end{proposition}

\begin{proof}
Suppose that statement (1) holds. Let $a \in \cA$ with $a^* = a$, 
and let $\phi \in D_2$
so that $\phi = \mu - \nu$ for some $\mu, \nu \in S(\cA)$. Then
\[
|\phi(a)| = |(\mu - \nu)(a)| \leq d^L(\mu, \nu) L(a) \leq 2rL(a)  .
\]
Since 
$\phi(1_\cA) = 0$, this inequality holds whenever $a$ is replaced
by $a - s1_\cA$ for some $s \in \bR$. Since
this is true for all $\phi \in D_2$ and since $a^* = a$, it follows that 
statement (2) holds.

\noindent
Conversely, suppose that statement (2) holds. 
Then for any $\mu, \nu \in S(\cA)$ and any self-adjoint
$a\in \cA$ with $L(a) \leq 1$ we have
\[
|\mu(a) - \nu(a)| = |(\mu - \nu)(a - s1_\cA)| \leq 2\|a - s1_\cA\|
\]
for any $s \in \bR$. Since $a^* = a$, it follows that $|\mu(a) - \nu(a)| \leq 2L^0(a)$.
Statement (1) follows from this.
\end{proof}

\begin{definition}
\label{raddef}
The smallest $r$ for which statements (1) and (2) 
just above hold is
called the \emph{radius} of $(\cA, L)$, and $2r$ is
called its \emph{diameter}.
\end{definition}

\begin{proposition}
\label{ordrad}
If $(X, d)$ is a compact metric space whose radius
is $r$, then the radius, as
defined above, of the compact C*-metric
space $(\cA = C(X), L^d)$, is $r$.
\end{proposition}

\begin{proof}
Given any $\bR$-valued $f \in \cA$, let $x_M$ and $x_m$ be
points where $f$ takes its maximum and minimum values. Than
\[
L^0(f) = (f(x_M) - f(x_m))/2  \leq L^d(f) d(x_M, x_m)/2  \leq rL^d(f)  .
\]
Thus $L^0(f) \leq rL^d(f)$, and the radius of $(\cA, L^d)$ is
no greater than $r$. Since $X$ is compact, we can find
$x,y \in X$ such that $d(x,y) = 2r$. Define $f$ by
$f(z) = d(y, z)$ for any $z \in X$. Then $L^d(f) \leq 1$
and $L^0(f) \geq |f(x) - f(y)|/2 = r$, so that the radius of 
$(\cA, L^d)$ is no less than $r$.
\end{proof}
 

\section{States, and distances between them}
\label{trancos}

We will want to explore what the metric on
probability measures is for the
case of sum-metrics on a product of a finite number
of finite metric spaces, in preparation 
for doing the same in the quantum situation. 
But in this section we first collect information about the
metric on states for general compact metric spaces,
in a way that will apply to the quantum situation. We
treat states for sum-metrics in the next section.

So let $(X, d)$ be a compact metric space.
Notice that
the nullspace of $L^d$ is $\bC 1$. Let $Lip^0_d(X)$ be the
quotient space $Lip_d(X)/\bC 1$. 
(Alternatively, $Lip^0_d(X)$ is often defined as the
subspace of functions $f \in Lip_d(X)$ such that $f(x_*) = 0$
for some fixed base-point, but
this does not have a very natural non-commutative version.)
Then $L^d$ determines
a norm on $Lip^0_d(X)$, which we will denote again by $L^d$. 
It is well-known that $Lip^0_d(X)$
is complete for this norm (see \cite{Wvr4}, though Weaver's
terminology is somewhat different from ours). For an $f \in Lip_d(X)$
we will denote its image in $Lip^0_d(X)$ again by $f$.

We seek to express everything in terms of the algebra $\cA=C(X)$,
and so we use the definition of
the Kantorovich metric, $d^{L^d}$, on the set $Prob(X)$ of probability
measures on $X$, in its well-known dual form. This is
often called the Kantorovich-Rubinstein metric \cite{Edw1, Edw2}.
Probability measures are then 
viewed as linear functionals (Radon measures) on $C(X)$,
and we define the metric, $d^{L^d}$, on $Prob(X)$ by
\[
d^{L^d}(\mu, \nu) = \sup \{|\mu(f) - \nu(f)|: 
f \in Lip^0_d(X) \ \ \mathrm{and} \ \   L^d(f) \leq 1\}
\]
for any $\mu, \nu \in Prob(X)$. This is exactly equation
\ref{distance} as applied to $L^d$.

By means of the Arzela-Ascoli theorem one finds 
that the unit $L^d$-ball of $Lip^0_d(X)$ is compact for 
the norm $L^0$ on $Lip^0_d(X)$. 
As one consequence of Theorem \ref{predual} below, we will obtain
a proof of the well-known fact that this implies that the metric $d^{L^d}$
determines the weak-$*$ topology on $Prob(X)$, and so $Prob(X)$ is
compact
for this metric topology. 

\begin{notation}
\label{pd}
We let $M^0(X)$ denote the space of all $\bC$-valued
Radon measures on $X$ that send $\bC 1_\cA$ to 0.
Thus it is the dual Banach space of $\cA^0$, and is equipped
with the norm dual to the norm $L^0$ on $\cA^0$.
We denote this dual norm by $\|\cdot\|_1$. It is 
just the usual total-variation norm, restricted to $M^0(X)$.
\end{notation}

 Because of the Jordan decomposition
for measures,  we can view
$M^0(X)$ as the vector space of (possibly $\bC$-valued)
linear functionals that is spanned by the
$\mu - \nu$'s for which $\mu$ and $\nu$ are probability measures.

Let $r$ be the radius of $(X, d)$.
From Propositions \ref{ordrad} and \ref{diamet}
we see that for every $\bR$-valued $f \in Lip^0_d(X)$ we have
$L^0(f) \leq rL^d(f)$. Consequently, every linear functional 
on $Lip^0_d(X) \subseteq C(X)$
that is continuous for the norm $L^0$ is also continuous
for the norm $L^d$. Thus $M^0(X)$ consists of linear
functionals that are continuous for the norm $L^d$, and
so we can also equip $M^0(X)$ with the
dual norm, $\| \cdot \|_d $, to $L^d$ (restricted to $M^0(X)$, so that,
\begin{equation}
\label{dnorm}
\|\phi\|_d = \sup \{ |\phi(f)|: \  L^d(f) \leq 1\}  
\end{equation}
for any $\phi \in M^0(X)$. 
We then see that $d^{L^d}$ is just the restriction of the norm
$\| \cdot \|_d $ to the elements of $M^0(X)$ of the form
$\mu - \nu$ for $\mu, \nu \in S(\cA)$.

As explained well in \cite{Wvr4, Wvr3, Var2}, in general
$(M^0(X), \| \cdot \|_d $) is not complete,
and its completion is not the full dual space
of the Banach space $ (Lip^0_d(X), L^d)$,
and in fact $M^0(X)$ is best viewed as
a dense subspace of
a \emph{pre-dual} for $Lip^0_d(X)$. But we will not need
to use this fact.

We now place these results in a more general context, which is
in particular, pertinent to
the non-commutative case. Our vector spaces can be over either
$\bR$ or $\bC$, although later we will see that for certain aspects there is a 
significant difference between these two cases.
For any normed vector space $Z $ we denote its closed unit ball by $B_Z$, 
and its dual by $Z'$. The next theorem is basically well-known,
but I have not seen it assembled in the way we give here. 
The part of the theorem concerning preduals is related to theorem 18 of \cite{Dxm4}.
When our theorem is applied to the discussion above, 
$V$ would be $\cA^0 =C(X)/\bC 1$, with $L^0$ as norm, and $W$
would be $Lip^0_d(X)$ with norm  $L^d$. 
Then $V'$ would be $M^0(X)$.

\begin{theorem}
\label{predual}
Let $\{V, \|\cdot\|_V\}$ be a normed vector space, and let $W$ be a dense 
subspace of $V$. Let $\| \cdot \|_W$ be a norm on $W $ for which there is a 
constant, $r \in \bR$, such that $\|w\|_V \leq r \|w\|_W$ for all $w \in W$.
Then each $\phi \in V'$ can be viewed as a continuous linear functional on $W$,
and $\|\phi\|_{W'} \leq r \|\phi\|_{V'}$, so that $B_{V'} \subseteq r B_{W'}$. 

Suppose that $B_W$ is totally bounded for $\|\cdot\|_V$. Then
the topology on $B_{V'}$ from the norm $\|\cdot\|_{W'}$ coincides
with the restriction to it of the weak-$*$ topology of $V'$ as the
dual of $V$ (for which $B_{V'}$ is compact). 

Furthermore, let $F$ be the subspace of $W'$ consisting
of those elements whose restriction to $B_W$ is continuous for
its topology from $\|\cdot \|_V$. Then $F$ is a predual for 
the completion of $W$, 
and $V'$ is  dense in $F$ for the norm $\|\cdot\|_{W'}$.
 \end{theorem}
 
\begin{proof}
We show first that for the situation in the first paragraph, the
 topology on $B_{V'}$ from the norm $\|\cdot\|_{W'}$ is stronger than
 the restriction to it of the weak-$*$ topology of $V'$ as the
dual of $V$. For a given $\phi \in  B_{V'}$ let 
\[
\cO(\phi, \{v_j\}, \epsilon) =\{\psi \in  B_{V'}: |(\phi - \psi)(v_j)| < \epsilon, \ \ 
 j = 1, \dots , n\}  ,
\]
 an arbitrary element of a neighborhood base for $\phi$ for the weak-$*$ topology of $V'$
 restricted to $B_{V'}$.
 We seek a $d \in \bR$ such that the open 
 ball $B_{V'}(\phi, d) = \{\psi \in  B_{V'}: \|\phi - \psi\|_{W'} < d\}$ about $\phi$ is contained in
$\cO(\phi, \{v_j\}, \epsilon) $. Let $M = \max\{\|v_j\|: j = 1,\dots , n\}$, and let
$d = \epsilon/rM$. Then we see that for any $\psi \in B_{V'}(\phi, d)$ and each $j$ we 
have $|(\phi - \psi)(v_j)| < \epsilon\}$, as needed.   

Now suppose that $B_W$ is totally bounded for $\|\cdot\|_V$. We show then that 
 the topology on $B_{V'}$ from the norm $\|\cdot\|_{W'}$ is weaker than
 the restriction to it of the weak-$*$ topology of $V'$ as the
dual of $V$ (so these topologies coincide in view of the previous paragraph).
For a given $\phi \in  B_{V'}$ and a $d > 0$, let 
$B_{W'}(\phi, d) = \{\psi \in  B_{V'}: \|\phi - \psi\|_{W'} < d\}$,  
 an arbitrary element of a neighborhood base for $\phi$ for the norm of $W'$.
 We want to show that this ball contains an open neighborhood of $\phi$
 for the weak-$*$ topology of $V'$. Since $B_W$ is totally bounded for $\|\cdot\|_V$,
 we can find $w_1, \dots, w_n$ in $B_W$ such that the $\|\cdot \|_V$-balls
 about them of
 radius $d/3$ cover $B_W$. Set
 \[
\cO(\phi, \{w_j\}, d/3) =\{\psi \in  B_{V'}: |(\phi - \psi)(w_j)| < d/3, \ \ j = 1, \dots , n\}  ,
\]
a weak-$*$ open neighborhood of $\phi$. Consider $w \in B_W$. There is
a $k$ such that $\|w - w_k\|_V < d/3$. Then for any $\psi \in \cO(\phi, \{w_j\}, d/3)$
we have
\begin{align*}
|(\phi & - \psi)(w)| \\
&\leq |\phi(w) - \phi(w_k)| + |\phi(w_k) - \psi(w_k)| + |\psi(w_k) - \psi(w)|  \\
&\leq \|\phi\|_{V'}\|w - w_k\|_V +  |(\phi - \psi)(w_k)| + \|\psi\|_{V'}\|w_k - w\|_V    < d   .
\end{align*}
Thus $\|\phi - \psi\|_{W'} < d$, so that $\psi \in B_{W'}(\phi, d)$ as needed.

We now show that the completion of $W$ is a dual space, assuming as above
that $B_W$ is totally bounded for $\|\cdot\|_V$. 
When we take the completions of $V$ and $W$, the conditions of the theorem
continue to hold. So we now assume that $V$ and $W$ are complete.
Of course a predual, $F$, for $W$ must
be a closed subspace of $W'$. Let $F$ consist of all the elements of $W'$
whose restriction to $B_W$ is continuous for the (compact) topology on $B_W$ 
from $\|\cdot\|_V$. Then $F$ is a closed subspace of $W'$.
Clearly $V' \subseteq F$, so $F$ separates the points of $W$. Define a linear
map, $\s$,  from $W$ into $F'$ by $\s(w)(f) = f(w)$. Then $\s$ is injective, and
$|\s(w)(f)| \leq \|f\|_{W'}\|w\|_W$, so that $\|\s\| \leq 1$. 
In particular, $\s(B_W) \subseteq B_{F'}$. Since the composition of $\s$
with the linear functional on $F'$ determined by any element of $F$ is, from the 
definition of $F$, obviously
 continuous on $B_W$ for its compact topology from $\|\cdot\|_V$, it follows that $\s$ 
 on $B_W$ is continuous
 for the weak-$*$ topology on $F'$. It follows that $\s(B_W)$ is compact for the
 weak-$*$ topology. 
 
 Since  $\s(B_W)$ is convex, the Hahn-Banach separation
 theorem tells us that
if $\s(B_W)$ were not all of $B_{F'}$ there would be a $\theta \in B_{F'}$ and
a weak-$*$ linear functional separating $\theta$ from $\s(B_W)$.  
But every weak-$*$ linear functional on $F'$ comes from an element of $F$.
So there would be an $f \in F$ such that $f(w) \leq 1 < \theta(f)$ for every
$w \in B_W$. But the first inequality says that $\|f\|_{W'} \leq 1$, and so
the second inequality says that $\|\theta\|_{F'} > 1$, contradicting
$\theta \in B_{F'}$. Thus $\s$ is an isometry from $W$ onto $F'$.

Finally, suppose that there is a $\theta \in F'$ such that 
$\theta(\phi) = 0$ for all $\phi \in V' \subseteq F$.
We have just seen that every element of $F'$ comes from an element of $W$,
so there is a $w \in W$ such that $0= \theta(\phi) =  \phi(w)$ for all $\phi$.
It follows that $w = 0$, so that $\theta = 0$. Then the Hahn-Banach theorem
tells us that the closure of $V'$ in $F'$ for the norm $\|\cdot\|_{W'}$ is all of $F'$.
\end{proof} 

I do not know what can be said about the uniqueness of the predual in general.
For compact metric spaces the predual is unique -- see 
\cite{Wvr3} or section 3.4 of \cite{Wvr4}.

In \cite{Wvr4}, and elsewhere, the completion of $M^0(X)$ 
for the norm $\| \cdot \|_d $ dual 
to $L^d$ is denoted by something like $AE_d(X)$ and called the
Arens-Eells space for $(X, d)$, after the names of the authors 
of the first paper to study these spaces \cite{ArE}.
In some other parts of the literature
this Banach space is denoted instead by something 
like $\cF_d(X)$ and called the
``Lipschitz-free'' space for $(X, d)$. There is a quite substantial literature
exploring the Banach-space properties of these spaces (see
\cite{DKO, OstO} and the references they contain), but many open questions remain.  

Now Theorem \ref{predual} applies equally well in the non-commutative
setting. Let $\cA$ be a unital C*-algebra and let $L$ be a seminorm on
$\cA$ that satisfies conditions (1) and (2) of Definition \ref{deflip}.

\begin{notation}
\label{ncpd}
We let $M^0(\cA)$ denote the space of all (continuous)
linear functionals on $\cA$ that send $\bC 1_\cA$ to 0.
Thus it is the dual Banach space of $\cA^0$, and is equipped
with the norm dual to the norm $L^0$ on $\cA^0$.
We denote this dual norm by $\|\cdot\|_1$. 
\end{notation}

 Then by the Jordan decomposition,
$M^0(\cA)$ is the vector space of (possibly $\bC$-valued)
linear functionals that is spanned by the
$\mu - \nu$'s for which $\mu,\nu \in S(\cA)$.
As discussed after Notation \ref{Lzero}, 
the set $\{\mu - \nu: \mu, \nu \in S(\cA)\}$
coincides with the the ball $D_2$ of radius 2 in the
$\bR$-subspace of self-adjoint elements of $M^0(\cA)$.

Let $\cL_L = \{a \in \cA:L(a) < \infty\}$, and let
$\cL_L ^0 = \cL_L/\bC 1_\cA$. Then $L^0$ and $L$
can both be viewed as norms on $\cL_L ^0$, and $\cL_L ^0$
is an $L^0$-dense subspace of $\cA^0$.
Suppose there is an $r \in \bR^+$ such that we have
$L^0(a) \leq rL(a)$ for all $a \in \cA$. Then the restriction
to $\cL_L$ of every $\phi \in M^0(\cA)$ is continuous
for the norm $L$. Thus we can also equip $M^0(\cA)$ with the
dual norm, $\| \cdot \|_L $, to $L$ (restricted to $M^0(\cA)$, so that,
\begin{equation}
\label{ddnorm}
\|\phi\|_ L= \sup \{ |\phi(a)|: \  a \in \cA, \ \ L(a) \leq 1\}  
\end{equation}
for any $\phi \in M^0(\cA)$. 
We then see that $d^{L}$ is just the restriction of the norm
$\| \cdot \|_L $ to the elements of $M^0(A)$ of the form
$\mu - \nu$ for $\mu, \nu \in S(\cA)$. We are now in
position to apply Theorem \ref{predual}. We obtain:

\begin{corollary}
Let $\cA$ be a unital C*-algebra and let $L$ be a seminorm on
$\cA$ that satisfies conditions (1) and (2) of Definition \ref{deflip}.
Assume that  there is an $r \in \bR^+$ such that
$L^0(a) \leq rL(a)$ for all $a \in \cA$. Let notation be as
above. Then the radius of $S(\cA)$ for $d^L$ is no bigger than $r$. 

Let $B_L$ be the unit $L$-ball of $\cL_L^0$. If $B_L$ is totally 
bounded for $L^0$, then
the topology on the $\|\cdot\|_1$-unit ball of $M^0(A)$
 from the norm $\|\cdot\|_L$ coincides
with the restriction to it of the weak-$*$ topology of $M^0(A)$ as the
dual of $\cA^0$ (for which its unit ball is compact). Consequently,
the metric topology on $S(\cA)$ from $d^L$ coincides with
the weak-$*$ topology, and so is compact.

Furthermore, let $F$ be the subspace of $L$-continuous linear 
functionals on $\cL^0$ consisting
of those elements whose restriction to $B_L$ is continuous for
its topology from $\|\cdot \|_1$. Then $F$ is a predual for 
the completion of $\cL^0$, 
and $M^0(A)$ is  dense in $F$ for the norm $\|\cdot\|_L$.

\end{corollary}

We need next to consider what happens for sum-metrics.


\section{States and sum-metrics}
\label{sumcos}

We continue our preparations for studying the quantum situation
by considering here properties of $M^0(X)$ with its dual
norm $\|\cdot\|_{L^d}$ for 
the case of sum-metrics. When we concentrate on that norm, 
we will write $M^0_d(X)$.
But for sum-metrics the situation is considerably more complicated
when the distinction between dual and predual must be made. So
in this section we will eventually consider only finite sets and finite-dimensional
vector spaces, which is actually all we need for Hamming metrics.

 We first
consider a more abstract setting, which will also be useful in
the quantum situation, and which is well-known 
at least in
closely related situations involving norms rather than seminorms.
Let $W$ be a vector space (e.g. $Lip^0_d(X)$) , 
and let $\{\rho_i: i \in I\}$ be a finite collection of seminorms on 
$ W$ that collectively separate the points of $ W $, that is,
if $\rho_i(w) = 0 $ for all $i \in I $ then $w=0 $. Thus we can
define a norm, $\|\cdot\|_W $, on $W $ by
\begin{equation}
\label{sumnorm}
\|w\|_W = \max \{\rho_i(w): i \in I\}.
\end{equation}
For each $i \in I $ let $N_i$
be the null space of $\rho_i $, so that $\rho_i $ can be viewed 
as a norm on $W/N_i $. Notice that $N_i$ is a closed subspace of $W$ 
for $\| \cdot\|_W$. Let $\tilde \rho $ be the norm on 
$\oplus (W/N_i) $ defined by
\begin{equation}
\label{addanorm}
\tilde \rho(\{w_i\}) = \max_i \{\rho_i(w_i): i \in I\},
\end{equation}
where our notation does not distinguish between whether $w_i$
is viewed as being in $W$ or in $W/N_i$. Let $q_i $ be the 
quotient map from 
$W $ onto $W/N_i $,
and let $\pi$ be the injection of $W$ into $\oplus_i W/N_i $ given by
\begin{equation}
\label{inaject}
\pi(w) = \{q_i(w)\}. 
\end{equation}
From the definition above for the norm on
$W$ it is clear that $\pi$ is isometric.

Let $W'$ be the dual of $W$, with dual norm $\|\cdot\|_{W'}$,
and for each $i \in I$ let $N_i^\perp$ be the
annihilator of $N_i$ in $W'$. Then $N_i^\perp$ can be identified with
the dual of $W/N_i$. For each $i$ let $\rho_i^*$ be the norm on 
$N_i^\perp$  dual to the norm 
$\rho_i$ on $W/N_i $. We seek a formula for $\|\cdot\|_{W'}$
in terms of the $\rho_i^* $'s.
Notice then that $\oplus N_i^\perp$ can be identified with
the dual of $\oplus (W/N_i) $.
Let $\tilde \rho^*$ denote the norm on $\oplus N_i^\perp$ dual to the 
norm $\tilde \rho$ on $\oplus (W/N_i) $
defined in equation \eqref{addanorm}.
Then it is easily seen that
$\tilde \rho^*$ can be expressed in terms of the $\rho_i^*$'s by
\[
\tilde \rho^*(\{\phi_i\}) = \sum_i \rho_i^*(\phi_i)
\]
where $\phi_i \in N_i^\perp$ for each $i$ (just as the dual
of $\ell^\infty(X)$ for a finite set $X$ is $\ell^1(X)$).
The dual of the isometry $\pi $ of $W$ into $\oplus (W/N_i) $ is easily seen
to be the quotient map $\pi^*$ from $ \oplus N_i^\perp$ onto $W'$
defined by
\[
\pi^*(\{\phi_i\}) = \sum \phi_i  .
\]
Now suppose $M$ and $N$ are normed vector 
spaces and that $\s$ is an isometry
from $M$ into $N$, with dual operator $\s^*$ from $N^*$ onto $M^*$. 
Then a simple argument shows that the norm
on $M^*$ is the quotient norm from the norm on $N^*$ via $\s^*$.
 When we apply this fact to the isometry  
defined in equation \eqref{inaject}, we obtain:

\begin{proposition}
\label{dualanorm}

Let $W$ be a vector space over $\bC$ (or $\bR$), and let
$\{\rho_i: i \in I\}$ be a finite collection of seminorms on 
$ W $ that collectively separate the points of $ W $. Let 
$\|\cdot\|_W$ be the norm on $W $ defined by
\begin{equation}
\label{asumnorm}
\|w\|_W= \max \{\rho_i(w): i \in I\}.
\end{equation}
For each $i \in I $ let $N_i$
be the null space of $\rho_i $, so that $\rho_i $ can be viewed 
as a norm on $W/N_i $.
 Let $W'$ be the dual space of $W$, with dual norm $\|\cdot\|_{W'}$.
Let $N_i^\perp$ be the annihilator of $N_i$ in $W'$, so that
$N_i^\perp$ can be identified with the dual space of $W/N_i $.
For each $i$ let $\rho_i^*$ be the norm on 
$N_i^\perp$ dual to the norm 
$\rho_i$ on $W/N_i $. Then
\[
\|\phi\|_{W'} = \inf \{ \sum \rho^*_i(\phi_i) : \phi = \sum \phi_i\
\ \ \mathrm{with} \ \  \phi_i \in N_i^ \perp, \ \ i \in I \} 
\]
for each $\phi \in W'$.
\end{proposition}

We now apply this proposition to the situation of Proposition
\ref{summet}, but only for finite metric spaces, where there
is no distinction between the dual and the predual
of $ Lip^0(X, d_S)$. Thus we view $M^0_d(X)$ as the
dual of $ Lip^0(X, d_S)$, with dual norm $\|\cdot\|_d$, much as 
defined in equation \eqref{dnorm}. We let  $X = \Pi_i X_i$, where
each $(X_i, d_i)$ is a finite
metric space, and we let $d_S$
be the sum-metric on $X$.  Let $L^{d_S}$ be the
Lipschitz seminorm for $d_S$, viewed as a
seminorm on $Lip(X, d_S)$, so that according to
Theorem \ref{summet} we have
\[
L^{d_S}(f) = \max\{ L_i(f): i \in I\}
\]
for all $f \in Lip(X, d_S)$. To apply Proposition \ref{dualanorm}
we let $W = Lip^0(X, d_S)$ and we view $L^{d_S}$ as a norm
on $W$. We also view each of the $L_i$'s as a seminorm on
$W$ (which works since $X$ is finite). 
From Proposition \ref{nullj} we see that the null-space, $N_i$,
of each $L_i$ is, as vector space, the image in $W$ of 
$C(\tilde X_i)$, where
$\tilde X_i = \Pi \{X_j: j \neq i\}$, 
and where  $C(\tilde X_i)$ is
viewed as a subalgebra of $C(X)$. Then $L_i$ can be viewed
as a norm on the quotient space $C(X)/C(\tilde X_i)$.
We then let $N_i^\perp$ be the annihilator of $N_i$
in  $W'$. As notation for the norm on $N_i^\perp$ dual to $L_i$
we use $\|\cdot\|_i$ instead of $\rho^*_i$.
We are then exactly in the situation of Proposition \ref{dualanorm},
and it tells us that:

\begin{corollary}
\label{dualw}
With $ Lip^0(X, d_S)$ and $M^0_d(X)$ as above for finite $X$, and
with notation as just above, 
for each $\phi \in M^0_d(X)$ we have
\[
\|\phi\|_d =  
\inf \{ \sum \| \phi_i \|_i : \phi = \sum \phi_i\
\ \ \mathrm{with} \ \  \phi_i \in N_i^ \perp, \ \ i \in I \}   ,
\]
for $N_i^\perp = C(\tilde X_i)^\perp \subseteq M^0_d(X)$. 
Thus if
$\mu$ and $\nu$ are probability measures on $X$, then
\[
d^{L^{d_S}}(\mu, \nu) = \inf \{ \sum \| \phi_i \|_i : \mu - \nu = \sum \phi_i\
\ \ \mathrm{with} \ \  \phi_i \in N_i^ \perp, \ \ i\in I \}    .
\]
By the argument in Remark \ref{remar} one can require all the $\phi_i$'s 
to be self-adjoint.
\end{corollary}

For specific examples and $\mu$ and $\nu$'s, it can be an interesting challenge to calculate
$d^{L^{d_S}}(\mu, \nu)$,  especially if the index set $I$ is large.


\section{Differential calculi and Dirac-type operators for sum-metrics}
\label{dircom}

In this section we will see how to express the seminorms $L_j$
defined above
in terms of first-order differential 
calculi, and then in terms of Dirac-type operators,
so that they can be combined to apply to $L^{d_S}$.
(By ``Dirac-type operator" we are really referring to \emph{how}
an operator is used, namely to provide a metric, along the lines
described below. Any non-zero self-adjoint operator
can serve as a Dirac-type operator in many ways.)

Let us first recall (e.g. page 321 of\cite{GVF}) the 
construction of the universal first-order
differential calculus for any unital algebra $\cA$ over $\bC$, since
we will somewhat imitate this construction here.
Let $\tilde\O_\cA = \cA \otimes \cA$, viewed as an $\cA$-bimodule
in the evident way. We define a derivation, $\d$, on $\cA$ 
with values in $\tilde\O_\cA$ by
\[
\d a = 1 \otimes a - a \otimes 1 .
\]
The fact that $\d$ is a derivation means that it satisfies
the Leibnitz equality
\[
\d(ab) = \d(a)b + a \d(b)
\]
for $a,b \in \cA$. Then let $\O_\cA $ be
the sub-bimodule generated by
the range of $d$, spanned by elements of the
form $a\d (b)$ (``first-order differential forms''). The pair $(\O_\cA, d)$ 
is the universal first-order differential calculus for $\cA$.

When $X$ is a compact space and $\cA = C(X)$,
then $\tilde\O_\cA = C(X\times X)$ (we use the completion
of the tensor product), and then $\O_\cA$ can be 
identified with $C_\infty(E)$ where 
$E = \{(x,y): x,y \in X, \ x \neq y\}$. Indeed, for any unital
algebra $\cA$ there is a natural identification of $\O_\cA$
with the kernel of the multiplication map from 
$\cA \otimes \cA$ to $\cA$. But for $\cA = C(X)$, when
the algebraic tensor product $C(X) \otimes _{alg} C(X)$
is viewed as a
(dense) subalgebra of $C(X\times X)$, the multiplication
map is the same as restriction of functions to the diagonal, $\D$,
of $X \otimes X$. That map is onto $C(X)$, and continuous.
When extended to all of $C(X\times X)$, its kernel is exactly
the space of functions in $C(X\times X)$ that vanish on $\D$, 
that is, $C_\infty(E)$. (See the ``important example"
shortly after exercise 8.3 of \cite{GVF}.)
Notice that if $X$
is considered to be just a set, then $E$ can be viewed 
as the set of edges of the complete graph with $X$
as the set of vertices.

Returning to the setting of the previous section, we
treat first the case in which $I$ contains just two elements. To
ease the notation, we will denote the two compact spaces by
$X$ and $Y$, with metrics $d_X$ and $d_Y$. We will set
$Z =X \times Y$, and let $d_S$ be the sum-metric on $Z$
using  $d_X$ and $d_Y$. 
For $f \in Lip(Z, d_S)$ we set, much as earlier,
\[
L_1(f) = \sup\{\frac{|f(a,p) - f(b,p)|}{d_X(a,b)}: a,b \in X, p \in Y\}  .
\]
Note that $L_1(f)$ is finite because $f \in Lip(Z, d_S)$,
as seen from Theorem \ref{summet}.
To define a corresponding bimodule we first set  
\[
E = \{ (a,p,b) \in X \times Y \times X: a \neq b\} .
\]
Notice that $E$ is an open subset of $ X \times Y \times X$ and
so may not be compact. We let $\O = C_b(E)$, 
consisting of the bounded continuous functions on $ E $,
equipped with the supremum norm.
Note that $\O$ has an
evident $C(X)$-bimodule structure coming from the two
copies of $X$ in $E$, and thus it also has a $Lip(Z, d_S)$-bimodule
structure. We define a derivation, $\d_1$, on $Lip(Z, d_S)$
with values in $\O$ by
\[
\d_1 f(a, p, b)  = (f(a,p) - f(b,p))/d_X(a,b)  .
\]
We then see that
\begin{equation}
\label{deriv}
L_1(f) = \| \d f\|_\infty      
\end{equation}
for all $f \in Lip(Z, d_S)$.

More generally,  suppose that $X = \Pi_i X_i$, where each
$(X_i, d_i)$ is a compact
metric space,
and we let $d_S$
be the sum-metric on $X$. Then for each $i$ we have
$X = X_i \times \tilde X_i$ where $\tilde X_i = \Pi \{X_j: j \neq i\}$ 
as in the previous section. We can then apply the results just above
to obtain for each $i$ a normed bimodule over $C(X)$ and a
derivation of $C(X)$ into it. We can then form the direct sum
of these bimodules, with supremum norm, and the direct sum
of the derivations, to obtain an expression for $L^{d_S}$ like
equation \ref{deriv}.

Alternatively, we can express all of this in terms of a Dirac-type
operator. Choose (finite) Borel measures on $X$ and $Y$
of full measure,
and let $\mu$ be the corresponding product measure
on $X \times Y \times X$, restricted to $E$.
Let $\cH = L^2(E, \mu)$, and let $\pi$ be the (faithful) representation
of $C(Z)$ (and so of $Lip(Z, d_S)$) 
on $\cH$ defined by $(\pi_f \xi)(a,p,b) = f(a,p)\xi(a,p,b)$ 
for all $\xi \in \cH$. Define a (possibly unbounded) operator, $D_1$,
( for ``Dirac'') on $\cH$ by
\[
(D_1\xi)(a, p, b) = \xi(b,p,a)/d_X(a,b)   
\]
(with domain consisting of those $\xi$ for which $D_1\xi \in \cH$).
A quick calculation shows that for any 
$f \in C(Z)$ the formal commutator $[D_1, \pi_f]$
is given by
\[
([D_1, \pi_f]\xi)(a,p,b) =  ((f(b,p) - f(a,p))/d_X(a,b))\xi(b,p, a)  
\]
for any $\xi \in \cH$. From this we easily obtain:

\begin{proposition}
\label{dir1}
For notation as above,
the functions $f$ for which
$[D_1, \pi_f]$ is a bounded operator on its (dense) domain are exactly
the $f$'s for which $L_1(f) < +\infty$, and for these $f$'s we have
\[
\|[D_1, \pi_f]\| = L_1(f)   .
\]
\end{proposition}

We remark that often $D_1$ does not have compact resolvent,  and
thus often $(Lip(Z, d_S), \cH, D_1)$ is not a ``spectral triple'' as 
defined by Connes. 

We can now treat the general case in which $X = \Pi_i X_i$ and
$d_S$ is the sum-metric of the $d_i$'s. For each $i$ let
$\tilde X_i = \Pi \{X_j: j \neq i\} $ as done earlier, 
so that $X = X_i \times \tilde X_i$.
We let $\tilde d_i$ be the metric on $\tilde X_i$ that is the sum-metric
of the $d_j$'s for which $j \neq i$. Then 
$d_S$ is, according to Proposition \ref{sumadd}, the sum-metric
of $d_i$ and $\tilde d_i$. Let $L_i$ be the 
seminorm on $C(X)$ for the decomposition $X = X_i \times \tilde X_i$
using equation \ref{partial} and $d_i$, and
let $(\cH_i, \pi^i)$
be the representation of $C(X)$ constructed as above, with
Dirac operator $D_i$, so that 
\[
L_i(f) = \|[D_i, \pi^i f]\|   .
\]
It is easily seen that 
this $L_i$ coincides with the $L_i$ of equation \ref{partial},
so that
\[
L^{d_S}(f) = \max_i\{ L_i(f)\}  .
\]
Now let $\cH = \bigoplus \cH_i$, let $\pi = \bigoplus \pi_i$,
and $D = \bigoplus D_i$. It is then clear that
\[  
L^{d_S}(f) = \| [D, \pi(f) ] \|  
\]
for every $f \in Lip(X, d_S)$.

To proceed further we now need to specialize somewhat 
by using another aspect of the Hamming metric.


\section{Complete-graph metrics, and their Dirac operators}
\label{seccomp}

The Hamming metric involves setting the distance between
the different letters of any given one of its 
alphabets to always be 1. This is
the metric on a complete graph for which all edges have
length 1. If $X$ is a finite set and $d$ is the corresponding
complete-graph metric, then the
corresponding Lipschitz seminorm, $L_c$ (``c'' for ``complete''), 
is defined by
\[
L_c(f) = \max \{|f(x)- f(y)|: x \neq y\}.
\] 
for all $f \in C(X)$. But this has meaning whenever $(X, d_X)$
is a compact metric space, although then this $L_c$ will not
be a C*-metric if $X$ is not finite, because requirement (4) of
Definition \ref{deflip} will fail. 
The seminorm $L_c$ is now given exactly by
he universal differential
calculus for the algebra $C(X)$, with $\O = C_\infty(E)$ for 
$
E = \{(x,y); x \neq y\}  .
$
(When $X$ is finite this is the edge-set of the complete graph.)
As before, the derivation, $\d$,
from $C(X)$ into $\O$ is defined by
\[
\d f(x,y)  = f(y) - f(x)  
\]
for any $(x,y) \in E$. 
It is clear that
\[
L_c(f) = \| \d f\|_\infty      
\]
for every $f \in C(X)$, as desired.

As before, $L_c$ can also
be represented by a Dirac-type operator, which
will be bounded in this case. For $\cH = L^2(E, \mu)$
where $\mu$ is the product with itself of
a finite Borel measure of full support on $X$, the 
Dirac operator, $D$, is defined by
 $D\xi(x, y) = \xi(y,x)$.
 
But there is another way of viewing this situation, which we will see
is more useful for obtaining the quantum Hamming metric.
Let $f \in C(X)$ be $\bR$-valued, let $x_M$ and $x_m$ be points
where $f$ takes its maximum and minimum values,
and let $s = (f(x_M) +f(x_m))/2$. Then one sees quickly that
\[
L_c(f) = 2\|f - s1_X\|_\infty  .
\] 
Even more, one sees easily that $s$ is the constant for which
the expression on the right takes its minimum value. That
is 
\begin{equation}
\label{nearest}
L_c(f) = 2\inf \{\|f - t1_X\|_\infty: t \in \bR\},
\end{equation}
or, $s1_X$ is the closest function in $\bC 1_X $ to $f$.
But this latter expression makes sense for $\bC$-valued
functions, where we also allow $ t \in \bC$. That is,
we can define $L_q$ (``q'' for ``quotient'') on $C(X)$ by
\[
L_q(f) = \inf \{\|f - z1_X\|_\infty: z \in \bC\}  ,
\]
so that $L_q$ is the quotient norm on $C(X)/\bC1_X$
pulled back to $C(X)$. But this is exactly the seminorm $L^0$
that we have been using earlier. Thus
$L_c = 2L^0$ on real-valued functions. However,
these two $*$-seminorms, $L_c$ and  $2L^0$, 
can differ on $\bC$-valued 
functions. Here is an example, well-known for other
purposes (see example 1.35 of \cite{Wvr4}). 

\begin{example}
Let $X$ be a set with 3 points, all at
distance 1 from each other, and let $f$ be the $\bC$-valued
function that sends the 3 point to the 3 cube-roots of
1. Then the constant function which is closest to $f$ is easily
seen to be the constant function 0,  so that 
$2L^0(f) = 2\inf \{\|f - t1_X \|_\infty: t \in \bR\} = 2$, whereas $L_c(f) = \sqrt 3$.
\end{example} 

We can also find a Dirac operator for $L^0$ (and so for $2L^0$).
The construction is somewhat more complicated
than that for $L_c$,
but it is suggestive of what one might do in the
quantum case.
For simplicity we will discuss here only the case when
$X$ is finite. But the discussion we give has a good
generalization to the case when $C(X)$ is replaced
by any unital (possibly non-commutative and 
infinite-dimensional) C*-algebra -- see theorem 3.2
of \cite{R24} and the constructions in its proof.

\begin{proposition}
\label{project}
Let $X$ be a finite set, and let $\cH = \ell^2(X)$, with $C(X)$
represented on $\cH$ as pointwise multiplication operators.
Let $\cP$ be the set of all rank-one projections on $\cH$.
Then for every $f \in C(X)$ we have
\[
L^0(f) = \sup \{\|[P, \ f]\|\}: P \in \cP\}  .
\]
\end{proposition}
\begin{proof}
Note that for any $g \in C(X)$ we have 
$\|[P, \ g]\| = \|[P-(1/2)1_X, \ g]\| \leq \|g\|_\infty$,
since $\| P-(1/2)1_X\| = 1/2$.
 Then for any $P \in \cP$ and $z \in \bC$ we have
$\|[P, \ g]\| = \|[P, \ g-z1_X]\| \leq  \|g-z1_X\|_\infty$, and so
$\sup \{\|[P, \ g]\|\}: P \in \cP\}  \leq L^0(g)$.
Now fix $f \in C(X)$.
For ease of bookkeeping in obtaining the reverse inequality,
let us assume that $L^0(f) = 1$. The first basic result of linear
approximation theory tells us that then there is a linear
functional, $\phi$, on $C(X)$ such that $ \|\phi\| = 1 = \phi(f)$
and $\phi(1)= 0$. (Just consider the image of $f$ in $C(X)/\bC 1$, 
whose norm there is $L^0(f)$, so that by the Hahn-Banach theorem
there is a linear functional on $C(X)/\bC 1$ of norm 1 whose value
on the image of $f$ is $L^0(f)$, and pull this linear functional
back to $C(X)$.) Because $X$ is finite, $\phi$ can be viewed as
an element of $\ell^1(X)$ with $\|\phi\|_1 = 1$. 
Let $\xi = v|\phi|$ be the polar
decomposition of $\phi$, and then set $\xi = |\phi|^{1/2}$
and $\eta = v|\phi|^{1/2}$. Then $\xi, \eta \in \cH$, and
$\|\xi\|_2 = 1 = \|\eta\|_2$, and
$\phi(g) = \<g\xi, \eta\>$ for every $g \in C(X)$. From the
fact that $\phi(1) = 0$ we see that $\xi \perp \eta$. Let
$P$ be the rank-one projection along $\eta$, and notice 
that $P\xi = 0$. Then 
\[
\<[P,g]\xi, \eta\> = \<Pg\xi, \eta\> - \<gP\xi, \eta\> 
=\<g\xi, \eta\> = \phi(g)  
\]
for every $g \in C(X)$. In particular, 
$\<[P,f]\xi, \eta\> = \phi(f) = 1$, so that 
$\|[P,f]\| \geq 1= L^0(f)$, as needed.
\end{proof}

The rank-1 projections on $\cH$ constitute the projective
space, $\cP(\cH)$, over $\cH$. It is compact for its natural
topology. Choose a Borel measure of full support on $\cP(\cH)$
and form the Hilbert space $L^2(\cP(\cH), \cH)$ of
$\cH$-valued functions on $\cP(\cH)$. 
Let $C(X)$ act on $L^2(\cP(\cH), \cH)$ by letting each 
$f \in C(X)$ act on the copy of $\cH$ over each point
$P$ in its usual way.
Let $D$ be
the tautological function on $\cP(\cH)$ that to each
point $P \in \cP(\cH)$ assigns the operator $P$ acting
on the copy of $\cH$ over $P$, and view $D$ as
the operator on $L^2(\cP(\cH), \cH)$ 
of pointwise application of the operators $P$. 
From Proposition \ref{project} we see
that 
\[
L^0(f) = \| [D, f] \|
\]
for each $f \in C(X)$. Thus $D$ serves as a Dirac operator
for $L^0$. 

We remark that when $X$ is a finite set we saw that the
Dirac operator for $L_c$ acts on a finite-dimensional
Hilbert space, while the Dirac operator for $L^0$ 
given just above still
acts of a Hilbert space of infinite dimension. It is an
interesting question as to whether for finite $X$ 
there is in general
a Dirac operator for $L^0$ that acts on a Hilbert
space that is finite-dimensional.


\section{The definition of the quantum Hamming metrics}
\label{alg}

We now combine some of the results obtained above to
reformulate the Hamming metric in terms of the algebra
$C(X)$ without mentioning points, so that if we drop
the requirement that the algebra be commutative, we
obtain the quantum version, which provides our
definition of the quantum Hamming metric. But then
we will notice that the definition makes sense in a
considerably broader context. 

In this
section sets like $X$ are finite.
It is helpful to begin by considering the case in which
our index set $I$ contains only two elements, $1$ and $2$. 
We set $\cA_i = C(X_i)$ for $i = 1,2$, 
so that $\cA = C(X_1 \times X_2) = C(X_1)\otimes C(X_2) 
= \cA_1 \otimes \cA_2$.
Then
 \[
L_1(f) = \sup  \{|f(v_1, v_2)-f(w_1, v_2)|  \}
\]
as $v_1, w_1$ range over $X_1$ and $v_2$ ranges over $X_2$.
 Let $f$ be $\bR$-valued. Then by the argument preceding equation \ref{nearest}, 
 for each $v_2$ there is an $r_{v_2} \in \bR$ such that
 \[
 \sup \{ |f(v_1, v_2)-f(w_1, v_2)|  : v_1, w_1 \in X_1\}
 = 2 \| f(\cdot, v_2) - r_{v_2} 1_{\cA_1}\|_\infty .
 \]
 Define $h\in \cA_2=C(X_2)$ by $h(v_2) = r_{v_2}$ for each $v_2 \in X_2$. 
Then
we see that 
 \[
L_1(f) = 2 \inf \{\|f- 1_{\cA_1}\otimes k\|: k\in \cA_2\} = 2\|f\|_{\cA/\cA_2}
 \]
 (because we can use $k = h$),
where $ \|f\|_{\cA/\cA_2}$ is the pull-back to
$\cA$ of the quotient norm on $\cA/\cA_2$. 
   
Recall that $L_H$ was provisionally defined shortly
after remark \ref{remar}.
A proof by induction on the size of $I$ then gives:

\begin{theorem}
For each $i \in I$ let $X_i$ be a finite set, and let $X = \prod_i X_i $,
so that $C(X) = \otimes_i C(X_i)$.
For each $i \in I$ 
let $\tilde X_i = \prod\{X_j: j \neq i\}$, 
so that $C(X) =  C(X_i) \otimes C(\tilde X_i)$.  
Define a seminorm $L_i$ on $C(X)$ by 
\[
L_i(f) = 2\|f\|_{C(X)/C(\tilde X_i)}
\]
for any $f \in C(X)$.
Then for any $\bR$-valued $f$ we have
\[
L_H(f) = \max \{L_i(f): i \in I\}  .
\]
\end{theorem}

Notice that the $L_i$'s above are defined just in terms of the algebras,
with no mention of points. Thus we can drop the assumption that
the algebras are commutative, and take the above formulation as
the \emph{definition} of the quantum Hamming metric.

\begin{definition}
\label{sham}
For each $i \in I$ let $\cA_i$ be a 
matrix algebra $M_{n_i}(\bC)$, and let $\cA = \otimes_i \cA_i $.
For each $i$
let $\tilde \cA_i = \otimes\{\cA_j: j \neq i\}$, 
so that $\cA =  \cA_i \otimes \tilde \cA_i$ 
when $\cA_i$ and $\tilde \cA_i$ are viewed as subalgebras
of $\cA$ in the evident way.
Define a $*$-seminorm, $L_i$, on $\cA$ by 
\begin{equation}
\label{defLi}
L_i(a) = 2\|a\|_{\cA/ \tilde \cA_i} = 2\inf\{\|a - c\|: c \in \tilde \cA_i\}
\end{equation}
for any $a \in \cA$. 
Then define a $*$-seminorm, $L_{qH}$, on $\cA$ by 
\begin{equation}
\label{hamdef}
L_{qH}(a) = \max \{L_i(a): i \in I\}  .
\end{equation}
This $L_{qH}$ is our \emph{quantum Hamming metric} on $\cA = \otimes_i \cA_i $.
\end{definition}

This definition corresponds to proposition 8 in \cite{DMTL}, except that
we do not restrict $a$ to being self-adjoint. The above definition immediately
makes sense when the $ \cA_i $'s are arbitrary finite-dimensional
C*-algebras. Even in this slightly more general case
we will see in Section \ref{hams} that $L_{qH}$ is a C*-metric that 
is strongly Leibniz.    

But on examining the above definition, we see that it has meaning even
when the $\cA_i$'s are arbitrary, perhaps not finite-dimensional, 
unital C*-algebras, once we clarify that we
will always use the minimal C*-tensor product \cite{KR2}. Without
further comment we will use the fact that the minimal C*-tensor product
is associative and respects sub-algebras and product states 
in the ways one would expect, as explained in 
section 11.3 of \cite{KR2}. We will never need to take quotients
by 2-sided \emph{ideals}, and so we need not be concerned with the
fact that the minimal C*-tensor product does not always work for quotients by ideals 
in the ways one would expect. See section 3.7 of \cite{BrO}.

It is clear that in this general case $L_{qH}$ will
still be a bounded seminorm on $\cA$. We will see that it
produces a finite metric on $S(\cA)$. But unless all the $\cA_i$'s
are finite-dimensional, the topology on $S(\cA)$ from this metric will not
agree with the weak-$*$ topology, and will not be compact.
Thus in this setting 
we will call  $L_{qH}$ the ``quantum Hamming seminorm",
reserving the term ``quantum Hamming metric'' for the case
in which the C*-algebras are finite-dimensional (but not
necessarily just full matrix algebras) so that $L_{qH}$ is
a C*-metric, as we will see in the next section. 


\section{The diameter of a quantum Hamming seminorm}
\label{hams}

We seek now a generalization of Proposition \ref{bound}. 
In the process we will obtain further properties of $L_{qH}$. 
We continue assuming that  the $\cA_i$'s 
are arbitrary unital C*-algebras,
and that $\cA$ is their tensor product.  We define the subalgebras
$\tilde \cA_i$ of $\cA$ exactly as in \ref{defLi}, and then define the
seminorms $L_i$ exactly as above. Thus we set
\[
L_i(a) = 2\|a\|_{\cA/\tilde \cA_i }   .
\]
We then define the seminorm
$L_{qH}$ exactly as above in \ref{hamdef}, and we refer to it as a 
quantum Hamming seminorm. We need to show that $L_{qH}$
is actually a norm not just a seminorm, on $\cA^0 = \cA/\bC 1_\cA$,
and that it is equivalent to the norm $L^0$.   
Since each of the $L_i$'s is strongly Leibniz (by theorem 3.2
of \cite{R24}), so is $L_{qH}$ (by proposition 1.2 of \cite{R21}).

Our generalization of Proposition \ref{bound}.
is more complicated than in the commutative case, and the
proof we give is strongly motivated by the proof of proposition 2 of \cite{DMTL}. 
(I thank Eleanor Rieffel for deciphering for me equation 17 of that
proof for $n=2$.) We need to totally order our index set $I$, so we now
take it to be the set of integers between $1$ and $n$.
For any $k \in I$ with $k \geq 2$ we set 
\[
B_k = \otimes _{i=1}^{k-1} \cA_i \quad \quad \mathrm{and} \quad 
\quad  C_k = \otimes _{i=k}^n \cA_i   \ ,
\]   
so that for each $k$ we have
$\cA = \cB_k \otimes \cC_k$. 
For any linear functional on $\cA$, say $\psi$, 
and for any unital C*-subalgebra $\cD$ of 
$\cA$, we denote the restriction of $\psi$ to $\cD$ by $\psi_\cD$, 
which is a linear functional on $\cD$.
To simplify notation we make the following conventions. For each $i$ let $1_i$
be the identity element of $\cA_i$, whereas if $K$ is a subinterval of $I$
of length $p$ then $I_p$ will denote the tensor product of the $1_i$'s for
$i \in K$. Whenever it is evident from the context, we will omit mentioning $K$,
and whenever it should not cause confusion, we will denote the C*-subalgebra
$I_{k-1} \otimes \cA_k \otimes I_{n-k}$ of $\cA$ simply by $\cA_k$. Then
$\tilde \cA_1 = 1_1 \otimes \cC_{2}$ and $\tilde \cA_n =  \cB_n \otimes 1_n$,
while $\tilde \cA_k =   \cB_k \otimes 1_k \otimes \cC_{k+1}$ for $2 \leq k \leq n-1$.

\begin{lemma}
\label{lemsum}
Let $\mu$ and $\nu$ be states on $\cA$, and let $\phi = \mu - \nu$. 
For each $k$ with $k \geq 2$ set
\[
\rho_k = \nu_{\cB_k} \otimes \mu _{\cC_k},
\] 
which is a state on $\cA$,
and for each $k$ with $2 \leq k \leq n-1$ set
\[
 \phi_k = \rho_k - \rho_{k+1}.
 \]
Then $\phi_k$ has the following properties for 
each $k$ with $2 \leq k \leq n-1$:
\begin{enumerate}
\item
The restriction of $\phi_k$ to 
$\tilde \cA_k = \cB_k \otimes 1_k \otimes \cC_{k+1}$ is the
0 functional.
\item
The restriction of $\phi_k$ to $\cA_k = I_{k-1} \otimes \cA_k \otimes I_{n-k}$ coincides
with $\phi_{\cA_k}$.
\end{enumerate}
Set $\phi_1 = \mu - \rho_2$. Then
\begin{enumerate}
\item
The restriction of $\phi_1$ to $\tilde \cA_1 = 1_1 \otimes \cC_{2}$ is the
0 functional.
\item
The restriction of $\phi_1$ to $\cA_1 = \cA_1 \otimes I_{n-1}$ coincides
with $\phi_{\cA_1}$.
\end{enumerate}
Set $\phi_n = \rho_n  - \nu$. Then
\begin{enumerate}
\item
The restriction of $\phi_n$ to $\tilde \cA_n = \cB_n \otimes 1_n$ is the
0 functional.
\item
The restriction of $\phi_n$ to $\cA_n =  I_{n-1} \otimes  \cA_n$ coincides
with $\phi_{\cA_n}$.
\end{enumerate}
Furthermore, $\phi = \sum \phi_k$ (a ``telescoping sum").
\end{lemma}

\begin{proof}
Let $k$ with $2 \leq k \leq n-1$ be given. Then for any $b \in \cB_k$
and $c \in \cC_{k+1}$ we have
\begin{align*}
\phi_k(b \otimes 1_k \otimes c) & = \rho_k(b \otimes 1_k \otimes c)
- \rho_{k+1}(b \otimes 1_k \otimes c)    \\
& = \nu_{\cB_k}(b)\mu_{\cC_k}(1_k \otimes c)
-  \nu_{\cB_{k+1}}(b \otimes 1_k )\mu_{\cC_{k+1}}(c)   \\
& =0   .
\end{align*}
But for any $a \in \cA_k$ we have
\begin{align*}
\phi_k(I_{k-1}\otimes a & \otimes I_{n-k})  = \rho_k(I_{k-1}\otimes a \otimes I_{n-k})
- \rho_{k+1}(I_{k-1}\otimes a \otimes I_{n-k})    \\
& =   \nu_{\cB_k}(I_{k-1})
\mu_{\cC_k}(a \otimes I_{n-k})
-  \nu_{\cB_{k+1}}(I_{k-1} \otimes a )\mu_{\cC_{k+1}}(I_{n-k})   \\
& = \phi(a) .
\end{align*}
The proofs for the other two cases are similar.
\end{proof}

Much as earlier, 
for each $a \in \cA$
we set $L^0(a) = \inf \|a - z1_\cA\|$, where 
$z$ ranges over $\bC$. Then note that just from its definition,
$L_i(a) \leq 2L^0(a)$ for each $i$, so that $L_{qH}(a) \leq 2L^0(a)$. 

\begin{theorem}
\label{diamer}
Let notation be as above. Then for any $a \in \cA$ we have
$|\phi_k(a)| \leq  L_k(a)$. If also $a^* = a$, then
\[
2L^0(a) \leq \sum_1^n L_k(a) \leq nL_{qH}(a)  \leq  2nL^0(a).
\]
In particular, it follows that $L_{qH}$ is a norm on $\cA^0$,  
and that the norms $L_{qH}$ and $L^0$ on $\cA^0$
are equivalent. Furthermore, the diameter of $(\cA, L_{qH})$ is
$\leq n$. 
\end{theorem}

\begin{proof}
Let $\mu$ and $\nu$ be states of $\cA$, and let
$\phi = \mu - \nu$. We apply Lemma \ref{lemsum}
to this $\phi$.
Notice that from the definition of the $\phi_k$'s in
Lemma \ref{lemsum}, $\|\phi_k\| \leq 2$ for each $k$.
We also see that
$\phi_k(\tilde \cA_k) = 0$ for each $k$.
Thus for any $a \in \cA$, if $d \in \tilde \cA_k$ then
\[
|\phi_k(a)| = |\phi_k(a-d)| \leq 2\|a-d\|  .
\]
On taking the infimum over $d$'s on the right side, we find that
$|\phi_k(a)|  \leq L_k(a)$, as desired.
Since also $L_k(a) \leq L_{qH}(a) \leq 2L^0(a)$, we obtain
\[
 \sum_1^n |\phi_k(a)| \leq   \sum_1^n L_k(a) \leq n L_{qH}(a) \leq 2nL^0(a).
\]

\noindent
Now assume that $a^* = a$, and let $\mu$ and $\nu$ 
be states of $\cA$ such that $\mu(a)$
and $\nu(a)$ are the maximum and minimum values of
the spectrum of $a$, so that $2L^0(a) = (\mu - \nu)(a)$. 
Then for $\phi = \mu - \nu$, we
have $\phi(a) = 2L^0(a)$. Since $\phi = \sum \phi_k$, it
follows that  $2L^0(a) \leq  \sum_1^n |\phi_k(a)| $, 
so that $2L^0(a) \leq  \sum_1^n L_k(a)$
as needed.

\noindent
We see that the norms $L_{qH}$ and $L^0$ on $\cA^0$
are equivalent on the self-adjoint part of $\cA$, but this inplies
that they are equivalent on all of $\cA$ (with slightly different
constants). The fact that $2L^0(a) \leq nL_{qH}(a)$ when $a^* = a$
says that the radius of $L_{qH}$ is no more than $n/2$, 
according to Definitiion \ref{raddef}. Thus
the diameter of $(\cA, L_{qH})$ is no more than $n$, generalizing what
happens in the commutative case as seen from Propositions
\ref{bound} and \ref{ordrad}.  
\end{proof}

The following theorem also generalizes what 
happens in the commutative case.

\begin{theorem}
\label{diam}
Let $\cA = \otimes_{i=1}^n \cA_i $ as above, but assume that none
of the $\cA_i$'s are 1-dimensional, i.e. isomorphic to $\bC$. Then
the diameter of $(\cA, L_{qH})$ is $n$.  
\end{theorem}

\begin{proof}
Since none of the $\cA_i$'s is 1-dimensional, we can choose for
each $i$ a self-adjoint element $a_i \in \cA_i$ such that $\|a_i\| = 1$,
and both +1 and -1 are in its spectrum. Let $\mu_i$ and $\nu_i$
be states on $\cA_i$ such that $\mu_i(a_i) = 1$ and $\nu(a_i)= -1$.
Let
\[
c_1 = a_1 \otimes I_{n-1} , \ \ c_2 =   I_1 \otimes a_2 \otimes I_{n-2}, \ \ \dots , \ \
c_n =I_{n-1} \otimes a_n     ,
\]
so that each $c_i$ is in the subalgebra $\cA_i$ of $\cA$.
Notice that $L_i(c_i) \leq 1$ for each $i$ since $\|c_i\| = 1$. 
Let $a = \sum_{i=1}^n c_i$. Then $L_i(a) \leq 1$ for each $i$ since
each $c_i$ is in the subalgebra $\cA_i$ of $\cA$. 
Thus $L_{qH}(a) \leq 1$.

Let $\mu = \otimes_{i= 1}^n  \mu_i$ and $\nu = \otimes_{i= 1}^n  \nu_i$,
product states on $\cA$. Then $\mu_{\cA_i} = \mu_i$ and  
$\nu_{\cA_i} = \nu_i$, so that  $\mu_{\cA_i}(a_i) = 1$ and 
$\nu_{\cA_i}(a_i) = -1$, for each $i$. 
Set $\phi = \mu - \nu$, and define
the $\phi_k$'s as in the above Lemma. Then by the properties stated
in the lemma, for each $i$
\[
\phi_i(a) = \phi_{A_i}(c_i) = (\mu - \nu)(a_i) = 2   ,
\]
so that $\phi(a) = 2n$. Since $\phi = \mu - \nu$, and $L_{qH}(a) \leq 1$,
this means that $d^L_{qH}(\mu, \nu) \geq 2n$, and so the diameter
of $S(\cA)$ is at least $2n$. Then from Propositiion \ref{diamet} 
and Theorem \ref{diamer} we see
that the diameter of  $(\cA, L_{qH})$ is exactly $n$.
\end{proof}


\section{The Kantorovich-Wasserstein metric on quantum Hamming states}
\label{states}

We continue to take the $\cA_i$'s  to be 
arbitrary unital C*-algebras, and as above we set 
$\cA = \otimes_i \cA_i $.
We equip $\cA$ with its quantum Hamming seminorm 
$L_{qH}$ defined above.
We will now use Proposition \ref{dualanorm} to 
express the corresponding ordinary metric on $S(\cA)$
from $L_{qH}$ in terms of the $L_i$'s, where $L_i$ is still defined
by equation \ref{defLi}.
View $L_{qH}$ as a norm
on $\cA^0$.
On $\cA^0$ we also have the norm $L^0$, but as seen in
Theorem \ref{diamer} these two norms are equivalent, and so define
the same topology on $\cA^0$. 
Recall that we denote the dual space of 
$\cA^0$ by $M^0(\cA)$. 
Let $\| \cdot \|_{qH}$ be the dual
norm on $M^0(\cA)$ determined  by $L_{qH}$
much as in equation \eqref{dnorm}, that is 
\[
\|\phi\|_{qH} = \sup \{|\phi(a)|: L_{qH}(a) \leq 1\}   ,
\]
where our notation does not indicate whether we view
$a$ as an element of $\cA$ or of $\cA^0$.
For every $i \in I$
we can view $L_i$ as a norm on $\cA/\tilde \cA_i$. 
The dual space of  $\cA/ \tilde \cA_i$.
can be identified with $\tilde \cA_i^\perp$, 
which is just the subspace of $M^0(\cA)$ consisting of functionals
that take value 0 on $\tilde \cA_i$. 
We denote the norm on $\tilde \cA_i^\perp$ that is
dual to $L_i$ by $\|\cdot \|_i$ . 
Since $L_{qH}(a) = \max\{L_i(a): i \in I\}$ by definition,
we are in position to apply Proposition \ref{dualanorm}
to obtain:

\begin{theorem}
Let notation be as above. Then for any $\phi \in M^0(\cA)$ we have 
\[
\|\phi\|_{qH} = \inf \{ \sum_i \|\phi_i\|_i : \phi_i \in \tilde \cA_i^\perp
\ \ \mathrm{and} \ \ \phi = \sum_i \phi_i \} .
\]
\end{theorem}

For any $\phi \in M^0(\cA)$ define $\phi^*$ by $\phi^*(a) = \overline{\phi(a^*)}$.
Note that $\|\phi^*\|_{qH} = \|\phi\|_{qH}$, 
so that  $\|(\phi + \phi^*)/2\|_{qH} \leq \|\phi\|_{qH}$.
Also, if $\phi \in \tilde \cA_i^\perp$ 
then $\phi^* \in \tilde \cA_i^\perp$.
Furthermore, if $\mu, \nu \in S(\cA)$ and we set 
$\phi = \mu - \nu$ then $\phi^* = \phi$.
Consequently, we obtain:  
\begin{corollary}
\label{stated}
Let notation be as above. Then 
for any $\mu, \nu \in S(\cA)$ we have
\[
d_{qH}(\mu, \nu) =  \inf \{ \sum_i \|\phi_i\|_i : \phi_i \in \tilde \cA_i^\perp,
\ \phi_i^* = \phi_i, \ \
 \mathrm{and} \ \ \mu - \nu = \sum_i \phi_i \} .
\]
\end{corollary}

For the case in which the $\cA_i$'s are full matrix algebras
this is basically the expression that  De Palma et al. take, in definition 7
of \cite{DMTL}, as their \emph{definition} of their
Wasserstein-1
metric on states, which  corresponds to the 
the quantum Hamming distance, 
though they express it in terms of density matrices.

It is an interesting challenge to find efficient methods for calculating
this distance between states. The most common situation in which 
this distance is
used is that in which each $\cA_i$ is $M_2(\bC)$ (a ``qbit") but the size
of $I$ can be quite large, e.g. $|I| = 500$.


\section{Dirac operators for quantum Hamming seminorms}
\label{dirh}

In this section we will assume that the C*-algebras that
we deal with are finite-dimensional. There are generalizations
of some of the results we give here to the case of infinite dimensional
C*-algebras \cite{R24}, but they are more technical,
and so we do not include them here. 

Thus let $\cA$ be a finite dimensional C*-algebra, necessarily 
unital, and let $ \cB$ be a C*-subalgebra of $\cA$ with
$1_\cA \in \cB$. Let $L_q$ be the pull-back to $\cA$ of
the quotient norm on $\cA/ \cB$. We seek to represent 
$L_q$ by a Dirac-type operator. We can construct one
by generalizing the construction used for Propositiion
\ref{project}.

Let $(\cH, \pi)$ be a finite dimensional faithful $*$-representation
of $\cA$, and view $\cA$ as a subalgebra of the algebra
$L^\infty(\cH)$ of all operators on $\cH$, with the operator norm. 
Notice that the distance from an element $a \in \cA$ to $\cB$
is independent of whether $a$ is viewed as an element of
$\cA$ or of $L^\infty(\cH)$. Thus we can, and we will, take 
$\cA$ to be all of $L^\infty(\cH)$.  

Then we let $\tau$ be the unique
tracial state on $\cA = L^\infty(\cH)$, and in the usual way
we define an inner product, $\<\cdot,\cdot\>_\tau$, on $\cA$,  
taken to be linear in the second variable, so that
$\<\xi, \eta\>_\tau = \tau(\xi^* \eta\>$ for all $\xi, \eta
\in \cA$. We denote the corresponding Hilbert-space
norm by $\|\cdot\|_2$.
(With more complicated bookkeeping one
can obtain versions of many of the results below if
$\tau$ is replaced by an arbitrary faithful state on $\cA$.)

Let $a$ be any element of $\cA$. Much as in the proof 
of Proposition \ref{project}, the first basic result of linear
approximation theory tells us that there is a linear
functional, $\phi$, on $\cA$ such that $ \|\phi\| = 1$
and $L_q(a) = \phi(a)$
while $\phi(b)= 0$ for all $b \in \cB$. 
Since $ \cA$ is finite dimensional, $\phi$ is 
continuous for the trace-norm, and so is represented 
by an element of $\cA$, which we will denote
by $f_\phi$. With that understanding, 
\[
\phi(c) = \< f_\phi, c\>_\tau = \tau(f_\phi^*c)
\]
for all $c \in \cA$. Since $\|\phi\| = 1$, we have $\|f_\phi\|_\tau = 1$,
where $\|\cdot\|_\tau$ is the trace-norm. Let $f_\phi = u|f_\phi|$ be
the polar decomposition of $f_\phi$, so that $u$ is a partial isometry
on $\cH$.
Let $\xi = |f_\phi|^{1/2}$ and $\eta = u|f_\phi|^{1/2}$, and notice
that $\xi^* = \xi$. Notice also that $\|\xi\|_2 = 1 = \|\eta\|_2$.
Most important, for any $c \in \cA$ we have
\[
\phi(c) =  \tau(f_\phi^*c) = \tau(\xi\eta^*c) = \<\eta, c\xi\>_\tau  .
\]
Since $\phi(b) = 0$ for all $b \in \cB$, we have
$\<\eta, b\xi\>_\tau = 0 $ for all $b \in \cB$. Let $P$
be the projection onto the subspace $\cB \eta$. Since this
subspace is clearly invariant under the (left) action of
$\cB$, the projection $P$ commutes with this action.
That is, $P$ is in the commutant of the action by $\cB$.
It is also clear that we have $P\eta = \eta$ while $P\xi = 0$. 
Consequently, for any $c \in \cA$ we have
\[
\< \eta, [P, c]\xi \>_\tau =  \<\eta, Pc\xi\>_\tau - \<\eta, cP\xi\>_\tau
=  \<\eta, c\xi\>_\tau = \phi(c)   . 
\]
In particular, we have
\[
L_q(a) = \phi(a) = \< \eta, [P, a]\xi \>_\tau   .
\]
Consequently, $L_q(a) \leq \| [P,a] \|$. But for any $c \in \cA$
and $b \in  \cB$ we have, since $P$ commutes with the
action of $\cB$,  
\[
\| [P,c] \| = \| [P, c-b] \| = \| [P - (1/2)I, c-b] \| \leq \|c-b\|  . 
\]
It follows that $\| [P,c] \| \leq L_q(c)$ for all $c \in \cA$,
and so, in particular, 
\[
L_q(a) = \| [P,a] \| .
\]
We see that we have obtained the following non-commutative
version of Proposition \ref{project}.

\begin{theorem}
\label{commutant}
Let $\cA = L^\infty(\cH)$ for a finite-dimensional Hilbert space
$\cH$, and let $\cB$ be a C*-subalgebra of $\cA$ that contains
$1_\cA$. Let $L_q$ be the pull-back to $\cA$ of the quotient norm
on $\cA/\cB$. Then for every $a \in \cA$ we have
\[
L_q(a) = \sup \{\| [P,a] \|: P \in \cP_\cB \}
\]
where $\cP_\cB$ is the set of projections in the
commutant of $\cB$ in $\cA$ of
rank no greater than $\dim(\cB)$.
\end{theorem}

This theorem is related to proposition 2.1 of \cite{Chr},
which, as explained there, is related to the Arveson
distance formula for nest algebras in \cite{Arv}.

To obtain a single Dirac-type operator for the situation,
for each $P \in  \cP_\cB$ let $\cH_P$ be a copy of
$\cH$, with the same representation of $\cA$ on
it, and let $\cK = \oplus \{\cH_P: P \in  \cP_\cB\}$, with 
corresponding direct-sum representation $\pi$ of $\cA$
on $\cK$. Let $D = \oplus \{ P: P \in  \cP_\cB\}$, 
a bounded operator on $\cK$. Then we clearly have
\[
L_q(a) = \| [D, \pi(a)]\|
\]
for all $a \in \cA$. Alternatively, $ \cP_\cB$, with metric from
the operator norm, is a compact space, and we can choose
a Borel measure of full support on it, and, much as in 
the comments after the proof of Proposition \ref{project}, we can
then form the Hilbert space of $\cH$-valued functions on
it. We then let $D$ be the operator on this Hilbert space
corresponding to the ``tautological'' function that assigns
to each point $P$ of $ \cP_\cB$ the value $P$. 
Then $D$ will serve as a Dirac operator for $L_q$.

It is easly seen that any seminorm that is obtained from
a Dirac operator in the way used above, is strongly
Leibniz, as defined in Definition \ref{deflip}.

It is an interesting question as to whether $L_q$ can
be determined by a Dirac-type operator on a finite-dimensional
Hilbert space. I suspect it can not in most cases, but I have
no proof of this.

We can apply Theorem \ref{commutant} to the quantum Hamming 
metrics $L_{qH}$, where $\cA = \otimes \cA_i$, with the
assumption that all the $\cA_i$ are finite-dimensional. 
For each $i$
we take $\tilde \cA_i$ as $\cB$ in that theorem. The commutant
of this $\cB$ will be exactly $\cA_i$. Thus:
\begin{corollary} 
With notation as above, for each $a \in \cA  = \otimes \cA_i$ we have
\[
L_{qH}(a) = \max \{L_i(a): i \in I\}
\]
where for each $i$ we have
\[
L_i(a) = \sup \{\| [P,a] \|: P \in \cP_{\cA_i} \},
\]
where the $\cP_{\cA_i} $ are constructed as in Theorem \ref{commutant}.
\end{corollary}

One can then express this in terms of a consolidated Dirac operator
along the lines discussed above. As seen earlier, this means
that $L_{qH}$ is strongly Leibniz.




}    

\end{document}